\documentclass[12pt,a4paper]{amsart}
\usepackage{times,amsmath,url,amscd}
\usepackage[all,dvips]{xy}
\usepackage{amssymb}
\usepackage{latexsym}
\usepackage{amsthm}

\newcommand{\Q}{\ensuremath{\mathbb{Q}}}
\newcommand{\Z}{\ensuremath{\mathbb{Z}}}
\newcommand{\N}{\ensuremath{\mathbb{N}}}
\newcommand{\C}{\ensuremath{\mathbb{C}}}

\newcommand{\Zpl}{\ensuremath{\Z_{(p)}}}

\newtheorem{thm}{Theorem}
\newtheorem{prop}[thm]{Proposition}

\newtheorem{lem}[thm]{Lemma}

\newtheorem{defn}[thm]{Definition}

\newtheorem{exs}[thm]{Examples}

\begin{document}

\title{Infinite Sums of Additive Unstable Adams Operations and Cobordism}

\author{M-J Strong}
\email{pmp04mas@sheffield.ac.uk}

\author{Sarah Whitehouse}
\email{s.whitehouse@sheffield.ac.uk}

\address{Department of Pure Mathematics, University of Sheffield, Sheffield S3 7RH, UK.}

\begin{abstract}
The elements of the ring of bidegree $(0,0)$ additive unstable operations
in complex $K$-theory can be described explicitly as certain
infinite sums of Adams operations.
Here we show how to make sense of the same expressions for complex cobordism $MU$,
thus identifying the ``Adams subring'' of the corresponding ring of
cobordism operations. We prove that the Adams subring is the centre of
the ring of bidegree $(0,0)$ additive unstable cobordism operations.

For an odd prime $p$, the analogous result in the $p$-local split setting
is also proved.
\end{abstract}

\date{$8^{\text{th}}$ December 2008}
\maketitle

\section{Introduction}
\label{intro}

In~\cite{gw} an injective ring map was defined from the ring of stable degree
zero operations in $p$-local $K$-theory to the corresponding ring of cobordism
operations. The main theorem identified the image of this map as the centre of
the target ring.

Here we show that that same thing happens in the additive unstable setting.
For a cohomology theory $E$, we work with additive unstable bidegree $(0,0)$
operations, that is natural transformations $E^0(-) \to E^0(-)$, where the functor
$E^0(-)$ is viewed as taking values in abelian groups.

For $E=KU$, the complex $K$-theory spectrum, all such operations can be
described in terms of Adams
operations, where certain specified infinite sums of Adams operations
are allowed. This goes back to work of Adams~\cite{adams-sln99}.

Since unstable Adams operations also exist for cobordism~\cite{wilson82, Kashiwabara94}, we can
consider the corresponding expressions for $MU$. The same infinite sums
converge and this allows us to define a ring map from the additive unstable
bidegree $(0,0)$ $K$-theory operations to the corresponding $MU$ operations.

The main work of this paper is devoted to showing that the image of this map
is precisely the centre of the target. The methods are close to those used
in the stable case, but suitably adapted to incorporate the Hopf ring
techniques necessary in the unstable case.
They exploit duality between
operations and cooperations for $K$-theory and for cobordism and they rely on the
fact that the operations under consideration are determined by their actions
on homotopy groups. In one respect the additive unstable case is simpler than
the stable situation: we are able to produce integral results directly
rather than by piecing together $p$-local results for each prime.

In the final section we prove the analogous result in the $p$-local split
setting.
\medskip

This paper is based on work in the Ph.D. thesis of the first author~\cite{strong},
produced under the supervision of the second author.

\section{Action on homotopy groups}
\label{sec2}

In this section we note the important fact that the operations we will be considering
act faithfully on homotopy groups.

First we introduce some notation. Our main reference for background on cohomology
operations is~\cite{bjw} and we adopt their notation and grading conventions. The
cohomology theories we will be concerned with
are complex $K$-theory $KU$ and complex cobordism $MU$. For an odd prime $p$, we
will also consider the Adams summand of $p$-local complex $K$-theory, which
we denote by $G$,  and the Brown-Peterson theory $BP$.

For a cohomology theory $E$, we denote by $\underline{E}_k$ the infinite loop spaces
in an $\Omega$-spectrum representing $E$. The unstable bidegree $(0,0)$ operations
of $E$-theory are given by $E^0(\underline{E}_0)\cong [\underline{E}_0, \underline{E}_0]$.
This is given the profinite topology and, as noted in~\cite{bjw}, it is complete with
respect to this topology for all of our examples.
Inside here are the \emph{additive} unstable bidegree $(0,0)$ operations $PE^0(\underline{E}_0)$,
which we will denote simply by $\mathcal{A}(E)$. Again, in all the theories we consider,
$\mathcal{A}(E)$ is complete with respect to the profinite filtration.

All our theories have good duality properties (see~\cite{bjw}). In particular,
operations are dual to cooperations: we have an isomorphism of $E^*$-modules
    $$
    E^*(\underline{E}_0)\cong \hom_{E^*}(E_*(\underline{E}_0), E^*).
    $$
The right-hand side is given the dual-finite topology: we filter by
    $$
    \ker\left(\hom_{E^*}(E_*(\underline{E}_0), E^*) \to \hom_{E^*}(L, E^*)\right),
    $$
where $L$ runs through finitely generated $E^*$-submodules of $E_*(\underline{E}_0)$.
Then the above isomorphism is a homeomorphism with respect to the profinite topology on the left-hand side
and the dual-finite topology on the right-hand side.

The additive operations $\mathcal{A}(E)$ are dual to $QE_*(\underline{E}_0)$,
 the indecomposable quotient of the cooperations for the $\star$-product:
    $$
    \mathcal{A}(E)=PE^*(\underline{E}_0)\cong \hom_{E^*}(QE_*(\underline{E}_0), E^*).
    $$
\smallskip

Let $\rm{Ab}_*$ denote the category of $\N$-graded abelian groups and
degree zero morphisms of abelian groups.
So
${\rm{Ab}}_*(M,N)$ denotes the degree zero homomorphisms  between two graded
abelian groups $M$ and $N$.
\smallskip

Given an unstable operation $\theta\in  E^0(\underline{E}_0)\cong [\underline{E}_0, \underline{E}_0]$, we may
consider the induced homomorphism of graded abelian groups
$\theta_*: \pi_*( \underline{E}_0)\to \pi_*( \underline{E}_0)$ given by the
action of $\theta$ on homotopy groups. Sending an operation to its action
on homotopy groups in this way gives a homomorphism of rings
    \begin{align*}
    E^0(\underline{E}_0) &\to \rm{Ab}_*\left(\pi_*( \underline{E}_0), \pi_*( \underline{E}_0)\right)\\
    \theta &\mapsto \theta_*.
    \end{align*}
We will consider the restriction of this map to the additive $E$-operations $\mathcal{A}(E)$ and denote this
by $\beta_E$:
    \begin{align*}
    \beta_E: \mathcal{A}(E) &\to {\rm{Ab}}_*\left(\pi_*( \underline{E}_0), \pi_*( \underline{E}_0)\right)\\
    \theta &\mapsto \theta_*.
    \end{align*}

\begin{prop}\label{faithfulaction}
For $E=MU$, $BP$, $KU$ or $G$,
the map 
    $$
    \beta_{E}:\mathcal{A}(E) \to 
    {\rm{Ab}}_*\left(\pi_*( \underline{E}_0), \pi_*( \underline{E}_0)\right)
    $$ 
is injective.
\end{prop}

\begin{proof}
As noted above, each of these theories has good duality, so
any $\theta\in \mathcal{A}(E)$ is uniquely determined by the corresponding
$E^*$-linear functional
    $$
    \bar{\theta}:QE_{\ast}(\underline{E}_{0}) \rightarrow E^{\ast}.
    $$

As $QE_{\ast}(\underline{E}_{0})$ and $E^{\ast}$ have no torsion, it is enough to
show that $\theta_{\ast}$ determines
    $$
    \bar{\theta}\otimes 1_{\Q}:QE_{\ast}(\underline{E}_{0})_\Q\rightarrow
    E^{\ast}_\Q,
    $$
where we are writing $M_\Q$ for $M\otimes \Q$.

By~\cite[12.4]{bjw}, the action of an operation on homotopy is given in terms
of the corresponding functional by
    $$
    \theta_{\ast}(t)=\overline{\theta}(e^{2h}\eta_{R}(t))
    $$
for $t\in E^{-2h}$. (Note that each of our theories $E$ has coefficients $E^*$ concentrated
in even degrees.)
Here $\eta_R: E^*\to QE_{\ast}(\underline{E}_{0})$ is the right
unit map and $e\in QE_{1}(\underline{E}_{1})$ is the suspension element.

Now for each of our theories $E$, every element of $QE_{\ast}(\underline{E}_{0})_\Q$ is an
$E^*_\Q$-linear combination of elements of the form $e^{2h}\eta_R(t)$,
where $t\in E^{-2h}$.
(See~\cite{bjw}; this may be proved by an inductive argument using the
relations in $QE_*(\underline{E}_*)$.)
It follows that $\overline{\theta}\otimes 1_{\Q}$ is
completely determined by $\theta_*$ as required.
\end{proof}

\section{Adams Operations}
\label{sec3}

We begin by discussing the definition and properties of unstable
Adams operations in $K$-theory and cobordism. 
We will denote our ground ring by $R$. Thus $R=\Z$ for $E=MU$ or $E=KU$ and 
$R=\Zpl$ for $E=BP$ or $E=G$.

The unstable Adams operations in $K$-theory $\Psi_{KU}^k$, for $k\in\Z$, were constructed
in~\cite{adams-vfs}. Unstable Adams operations for complex cobordism $MU$ and for the Brown-Peterson
theory $BP$ were defined by Wilson~\cite{wilson82} and also discussed by
Kashiwabara~\cite{Kashiwabara94}. These sources give us the following proposition.

\begin{prop}\label{adamsops}
Let $E=MU$, $BP$, $KU$ or $G$. For $k\in R$, there is an unstable Adams operation
$\Psi_E^{k}\in \mathcal{A}(E)$ such that
$(\Psi_E^{k})_*:\pi_{2n}(\underline{E}_0)\to \pi_{2n}(\underline{E}_0)$
is multiplication by $k^n$. 
These operations satisfy $\Psi_E^k\Psi_E^l=\Psi_E^{kl}$.
Furthermore $\Psi_E^k$ is
multiplicative. \qed
\end{prop}

For $KU$, all additive unstable bidegree $(0,0)$ operations can be described in 
terms of Adams operations.

\begin{thm}\cite{adams-sln99}
\label{Kbasis}
The topological ring $\mathcal{A}(KU)$ may be identified with the collection
of infinite sums $\{\sum_{n=0}^\infty a_n \sigma_n^{KU}\, |\, a_n\in \Z \}$,
where
    $$
    \sigma_n^{KU}=\sum_{k=0}^{n}(-1)^{n+k}\binom{n}{k}\Psi_{KU}^{k}.
    $$
These expressions are added termwise and multiplied using 
$\Psi_{KU}^k\Psi_{KU}^l=\Psi_{KU}^{kl}$.
\qed
\end{thm}

Explicit multiplication and comultiplication formulas for this topological
basis, as well as further results, can be found in~\cite{strong}.
\medskip

We now show that the corresponding infinite sums of Adams operations
are also defined for $MU$.
First we note some information about the Adams operations viewed
as functionals on the cooperations.

\begin{lem} \label{adamsMUopfunctional}
The Adams operation $\Psi^{k}_{MU}$ considered as a functional
    $$
    \overline{\Psi^{k}_{MU}}:QMU_{\ast}(\underline{MU}_0)\rightarrow MU^{\ast}
    $$
is determined by
    $$
    e^{2h}\eta_{R}(x) \ \mapsto \ k^{h}x \quad \textrm{for} \ x\in MU^{-2h}.
    $$
\end{lem}

\begin{proof}
As noted in the proof of Proposition~\ref{faithfulaction},
$\overline{\Psi^{k}_{MU}}\otimes 1_\Q$ determines $\overline{\Psi^{k}_{MU}}$  and it is enough
to specify this $MU^*$-linear map on elements of the form $ e^{2h}\eta_{R}(x)$ for
$ x\in MU^{-2h}$ since these generate $QMU_{\ast}(\underline{MU}_0)_\Q$
as a module over $MU^*_\Q$.

That these values are as claimed follows from the relation
    $$
    {\Psi^k_{MU}}_{\ast}(x)=\overline{\Psi^{k}_{MU}}(e^{2h}\eta_{R}(x))
    $$
for $x\in MU^{-2h}\cong \pi_{2n}(\underline{MU}_0)$ and the 
action of $\Psi^k_{MU}$ on homotopy groups.
\end{proof}

The following combinatorial lemma will be useful.

\begin{lem}\label{stirlingidentity}
For $m,n\geq 0$,
    $$
    \sum_{k=0}^{n}(-1)^{n+k}\binom{n}{k}k^{m}=
    n!\genfrac{\{}{\}}{0pt}{}{m}{n},
    $$
where $\genfrac{\{}{\}}{0pt}{}{m}{n}$ denotes a Stirling number of
the second kind. In particular, 
    \begin{align*}
    \sum_{k=0}^{m}(-1)^{m+k}\binom{m}{k}k^{m}&=m!,\\
    \sum_{k=0}^{n}(-1)^{n+k}\binom{n}{k}k^{m}&=0 \quad \text{if $n>m$}. 
    \end{align*}
\qed
\end{lem}

\begin{defn}
For $n\in\N$, define $\sigma_n^{MU}\in \mathcal{A}(MU)$ by
    $$
    \sigma_n^{MU}=\sum_{k=0}^{n}(-1)^{n+k}\binom{n}{k}\Psi_{MU}^{k}.
    $$
\end{defn}

\begin{prop}\label{infsumsMU}
The infinite sums $\sum_{n=0}^\infty a_n\sigma^{MU}_n\!$, where
$a_n\in\Z$, are well-defined operations in $\mathcal{A}(MU)$.
\end{prop}

\begin{proof}
By Proposition~\ref{adamsops}, we have the Adams operations $\Psi^k_{MU}\in \mathcal{A}(MU)$.
So clearly finite sums of the $\sigma^{MU}_n$ are well-defined operations
in $\mathcal{A}(MU)$. To see that the same is true for the infinite sums, it suffices
by completeness to show that $\sigma^{MU}_n\to 0$ as $n\to \infty$ in the profinite 
topology on $\mathcal{A}(MU)$.

Now $\mathcal{A}(MU)$ is homeomorphic to
$\hom_{MU^{\ast}}(QMU_{\ast}(\underline{MU}_0),MU^{\ast})$
with the dual-finite topology.
Since $QMU_{\ast}(\underline{MU}_0)$ and $MU^{\ast}$ are torsion-free, we have
an injective map
    $$
    \hom_{MU^{\ast}}(QMU_{\ast}(\underline{MU}_0),MU^{\ast}) \hookrightarrow
    \hom_{MU^{\ast}_\Q}(QMU_{\ast}(\underline{MU}_0)_\Q,
    MU^{\ast}_\Q),
    $$
given by
    $$
    f\mapsto f\otimes 1_{\Q}.
    $$
This is a homeomorphism to its image,
where the target is also endowed with the dual-finite
topology. Thus it is enough to show that $\overline{\sigma_n^{MU}}\otimes 1_\Q \to 0$
as $n\to\infty$.

Using Lemmas~\ref{adamsMUopfunctional} and~\ref{stirlingidentity},
for $x\in MU^{-2h}$, we have
    $$
    \overline{\sigma_{n}^{MU}}\left(e^{2h}\eta_{R}(x)\right)
    = \left(\sum_{k=0}^{n}(-1)^{n+k}\binom{n}{k}k^{h}\right)x
    = n!\genfrac{\{}{\}}{0pt}{}{h}{n} x.
    $$
Thus, $\overline{\sigma_{n}^{MU}}\left(e^{2h}\eta_{R}(x)\right)=0$ if $h<n$
and 
$\overline{\sigma_{n}^{MU}}\otimes 1_{\Q}$ is zero on the
$MU^{\ast}_\Q$-submodule of $QMU_{\ast}(\underline{MU}_0)_\Q$
generated by the finite collection of elements of the form $e^{2h}\eta_{R}(x)$
where $x$ runs through a $\Z$-basis of $MU^{-2h}$ and $h<n$.

Since $QMU_{\ast}(\underline{MU}_0)_\Q$ is generated as an
$MU^{\ast}_\Q$-module by $e^{2h}\eta_{R}(x)$ where $x$ runs 
through a $\Z$-basis of $MU^{-2h}$ and
$h\geq 0$, it
follows that $\overline{\sigma_{n}^{MU}}\otimes 1_{\Q}\rightarrow 0$ as
$n\rightarrow\infty$ in the dual-finite topology.
\end{proof}

\begin{prop}\label{kinject}
The map
    $$
    \iota: \mathcal{A}(KU)\to \mathcal{A}(MU)
    $$
given by
    $$
    \sum_{n=0}^\infty a_n\sigma^{KU}_n
    \mapsto\sum_{n=0}^\infty a_n\sigma^{MU}_n
    $$
is an injective ring homomorphism.
\end{prop}

\begin{proof}
Consider $\sum_{n=0}^{\infty}a_{n}\sigma_{n}^{MU}=
\iota\left(\sum_{n=0}^{\infty}a_{n}\sigma_{n}^{KU}\right)$ in
$\mathcal{A}(MU)$ and suppose $a_{m}\neq 0$ with $m$ minimal.
Since
    $$
    \sum_{n=0}^{\infty}a_{n}\sigma_{n}^{MU}=
    \sum_{n=0}^{\infty}a_{n}\sum_{k=0}^{n}(-1)^{n+k}\binom{n}{k}
    \Psi_{MU}^{k},
    $$
this operation acts on $\pi_{2m}(\underline{MU}_{0})\neq 0$ as
multiplication by
    $$
    \sum_{n=0}^{\infty}a_{n}\sum_{k=0}^{n}(-1)^{n+k} \binom{n}{k}k^{m}.
    $$

Since we have assumed that $a_{n}=0$ for $n<m$,
it follows from Lemma~\ref{stirlingidentity}
that $\sum_{n=0}^{\infty}a_{n}\sigma_{n}^{MU}$ acts on
$\pi_{2m}(\underline{MU}_{0})$ as multiplication by
    $a_{m}m!\neq 0$.
So $\sum a_{n}\sigma_{n}^{MU}$ is a
non-trivial operation in $\mathcal{A}(MU)$ and therefore $\iota$ is
injective.

It is easy to see that we have an algebra map:
the product of two
infinite sums is determined in both the source and the target by the products
of Adams operations.
\end{proof}

We note that the injective map $\iota$ above also respects the coalgebra
structure that we have on each side, since the comultiplication on
a general infinite sum is determined
by the fact that the Adams operations are
group-like.
\smallskip

We think of the image of $\iota$ as the ``Adams subring'' of $\mathcal{A}(MU)$. Our main
result (Theorem~\ref{mainthm}) is that this is the centre of $\mathcal{A}(MU)$.
We can prove one inclusion immediately.

\begin{lem}\label{imincentre}
The image $\textrm{Im}(\iota)$ is contained in the centre $Z(\mathcal{A}(MU))$.
\end{lem}

\begin{proof}
It is enough to show that the operations
$\sigma_{n}^{MU}$ commute with all elements of $\mathcal{A}(MU)$.
It is clear from the action of $\Psi^k_{MU}$
on homotopy that $\beta_{MU}(\Psi^k_{MU})$ commutes with
all elements of ${\rm{Ab}}_*\left(\pi_*( \underline{MU}_0), \pi_*( \underline{MU}_0)\right)$.
So the same holds for  $\beta_{MU}(\sigma_{n}^{MU})$. But by Proposition~\ref{faithfulaction},
$\beta_{MU}$ is injective, so
$\sigma_{n}^{MU}$ commutes with all elements of $\mathcal{A}(MU)$.
\end{proof}

\section{Diagonal operations and congruences}
\label{sec4}

\begin{defn}\label{diagonals}
Let $E=MU$ or $BP$.
Write $\mathcal{D}(E)$ for the subring of $\mathcal{A}(E)$ consisting of operations
whose action on each homotopy group $\pi_{2n}(\underline{E}_0)$ is multiplication
by an element $\lambda_{n}$ of the ground ring $R$. We call
elements of $\mathcal{D}(E)$ \emph{unstable diagonal operations}.
\end{defn}

\begin{lem}\label{centreindiag}
Let $E=MU$ or $BP$.
There is an inclusion $Z(\mathcal{A}(E))\subseteq \mathcal{D}(E)$.
\end{lem}

\begin{proof}
We note that there is an injection
$E^0(E)\hookrightarrow \mathcal{A}(E)$ 
from the stable degree
zero operations to the additive unstable bidegree $(0,0)$ operations,
given by sending a stable operation to its zero component. Indeed,
this map fits into a commutative diagram
    $$
    \xymatrix{
    E^0(E) \ar[r]
    \ar[dr]_{\alpha_E} & \mathcal{A}(E)\ar[d]_{\beta_E}\\
    &\rm{Ab}_*(\pi_*(\underline{E}_0),\pi_*(\underline{E}_0)),
    }
    $$
where $\alpha_E$ sends a stable operation to its action on $\pi_*(E)=\pi_*(\underline{E}_0)$.
But, as noted in~\cite{gw}, $\alpha_E$ is injective, so the map
$E^0(E)\to \mathcal{A}(E)$ is injective.

In~\cite[Proposition 14]{gw} particular Landweber-Novikov operations were exploited
in order to show that, for $E=MU$ and $E=BP$, a central stable operation has to act diagonally
on homotopy. Using the above inclusion, we consider the images of these
Landweber-Novikov operations
and then
exactly the same argument shows
that commuting with these elements forces
a central element of $\mathcal{A}(MU)$  or $\mathcal{A}(BP)$ to act diagonally
on homotopy.
\end{proof}

We continue to study the injective ring homomorphism 
$\iota: \mathcal{A}(KU)\to \mathcal{A}(MU)$ of Proposition~\ref{kinject}.
By Lemmas~\ref{imincentre} and~\ref{centreindiag}, we have
    $$
    \textrm{Im}(\iota)\subseteq Z(\mathcal{A}(MU))\subseteq \mathcal{D}(MU).
    $$

Our aim is to show that $\textrm{Im}(\iota)=\mathcal{D}(MU)$ and thus
$\textrm{Im}(\iota)=Z(\mathcal{A}(MU))$. The strategy is to characterize
$\mathcal{D}(MU)$ by a system of congruences and to compare this
with a system of congruences governing the $K$-theory operations
$\mathcal{A}(KU)$. For future use, we will also set up the corresponding congruences
for $BP$ and the Adams summand $G$.
\smallskip

For $E=MU, BP, KU$ or $G$, the relevant congruences arise as follows.
Let $\theta\in\mathcal{A}(E)$ be a diagonal operation, so that $\theta$ acts on
$\pi_{2n}(\underline{E}_0)$ as multiplication by
an element $\lambda_{n}$ of the ground ring $R$.
Of course, for $KU$ and $G$ all operations are diagonal.
 Consider the corresponding $E^*$-linear functional
$\bar{\theta} : QE_{*}(\underline{E}_0)\rightarrow E^{*}$.
We get a set of congruences which must be satisfied by the $\lambda_n$,
characterizing diagonal operations,
arising from $\bar{\theta}(x)\in E^{*}$ for all
$x\in QE_{*}(\underline{E}_0)$. For all these theories,
$QE_{*}(\underline{E}_0)$ is free as an $E^*$-module and
of course, we can let $x$ run through
a basis.

\begin{defn}
We write $S_{E}$ for the subring of $\prod_{n=0}^{\infty}R$
consisting of sequences $(\lambda_n)_{n\geq 0}$ satisfying this system of congruences.
\end{defn}

To be more explicit about these congruences we recall some further information about
the Hopf rings of these theories. Let $x^E\in E^2(\C P^\infty)$ be a choice of
complex orientation class, so that
$E^*(\C P^\infty)=E^*[[x]]$. (Later it will be convenient to choose
$x^{KU}=\varphi_*(x^{MU})$ and $x^G=\hat{\varphi}_*(x^{BP})$
where $\varphi:MU\to KU$ and $\hat{\varphi}:BP\to G$
are the standard maps of ring spectra.)
For any space $X$, there is a coaction map
    $$
    \rho: E^k(X)\to E^*(X)\widehat{\otimes}QE_*(\underline{E}_k);
    $$
see~\cite[6.26]{bjw}. The standard elements $b_{i}^{E}\in QE_{2i}(\underline{E}_2)$
are characterized by the property that
    $$
    \rho(x^E)=\sum_{i=0}^{\infty} b_{i}^{E} (x^{E})^i
        \in E^{\ast}(\C P^\infty)\widehat{\otimes}QE_*(\underline{E}_2)
        =QE_*(\underline{E}_2)[[x^E]].
    $$
Note that $b_0^{E}=0$ and $b_1^{E}=e^2$,
where $e$ is the suspension element.
\medskip

For $E=KU$, a completely explicit formulation
of the congruences, which can be found in~\cite[Theorem 4]{ClarkeBU}, is
    $$
    \frac{\sum_{k}(-1)^{n-k}\genfrac{\langle}{\rangle}{0pt}{}{n}{k}
    \lambda_{k}}{n!} \in\Z \qquad \textrm{for all} \ n\geq 0,
    $$
where the $\genfrac{\langle}{\rangle}{0pt}{}{n}{k}$ are Stirling numbers of
the first kind. We will abbreviate this system of congruences to
    $$
    C_n\cdot \lambda \in\Z \qquad \textrm{for all} \ n\geq 0,
    $$
where we are adopting vector notation and
    \begin{align*}
    C_n     &=\left(\frac{\sum_{k}(-1)^{n-k}\genfrac{\langle}{\rangle}{0pt}{}{n}{k}}{n!}\right)_{k\geq 0},\\
    \lambda &=(\lambda_k)_{k\geq 0}.
    \end{align*}

The coefficient ring $KU^*$ is given by $KU^*=\Z[u,u^{-1}]$ where $|u|=-2$.
As is usual, we write $v=\eta_R(u)\in QKU_0(\underline{KU}_{-2})$.
One finds the congruences above
by letting $x$ run through the $KU^*$-basis for $QKU_{*}(\underline{KU}_0)$
given by the elements $1$ and $b_n^{KU}v$ for $n>0$; see~\cite[Theorem 16.15]{bjw}.
In more detail, the periodicity of $K$-theory gives an isomorphism
of $KU^*$-modules
    $$
    QKU_*(\underline{KU}_0)\cong KU^*\otimes_\Z QKU_0(\underline{KU}_0)
    $$
and $QKU_0(\underline{KU}_0)$ may be identified with the ring of integer-valued
polynomials
    $$
    A=\{f(w)\in\Q[w]\,|\,f(\Z)\subseteq \Z \};
    $$
see~\cite{ClarkeBU, schwartz}. The explicit congruences above arise from the 
binomial polynomial basis for the ring of integer-valued polynomials and
the expansion of the
binomial polynomials in terms of Stirling numbers. Another way of
expressing the same thing is that $\pi_\lambda(b_n^{KU}v)=C_n\cdot\lambda$,
where $\pi_\lambda: QKU_*(\underline{KU}_0)\to \Z$ is the map determined rationally by
    $$
    u^{a}e^{2b}v^{b}\mapsto\lambda_{b}.
    $$

\begin{exs}
The first non-trivial congruences in this family are
    $$
    \begin{tabular}{cl}
     \vspace{0.2cm}
    \large{$\frac{\lambda_{2}-\lambda_{1}}{2} $ }              &$\in\Z$,\\
    \vspace{0.2cm}
    \large{$\frac{\lambda_3-3\lambda_{2}+2\lambda_{1}}{6}$}           &$\in\Z$,\\
    \large{$\frac{\lambda_4-6\lambda_3+11\lambda_{2}-6\lambda_{1}}{24}$} &$\in\Z$.
    \end{tabular}
    $$
\end{exs}
\medskip

\noindent 
Since all operations in $\mathcal{A}(KU)$ are diagonal and diagonal operations are precisely
characterized by the congruences, it is immediate
that we have an isomorphism of rings $\mathcal{A}(KU)\cong S_{KU}$, given by sending
an operation to its action on homotopy.
\medskip

Next we give further details of the congruences characterizing $\mathcal{D}(MU)$.
Consider the restriction of the map
    $$
    \beta_{MU}: \mathcal{A}(MU) \to {\rm{Ab}}_*\left(\pi_*( \underline{MU}_0), \pi_*( \underline{MU}_0)\right)
    $$
to diagonal operations $\mathcal{D}(MU)$. By the definition of $\mathcal{D}(MU)$, this restriction may be viewed as a map
    $$
    \beta_{MU|}:\mathcal{D}(MU)\rightarrow \prod_{n=0}^{\infty}\Z,
    $$
where we
implicitly identify the homomorphism given by multiplication by an
integer $\lambda$ on an abelian group with the integer $\lambda$.
\smallskip

Let $\theta\in \mathcal{D}(MU)$ with $\theta$ acting on
$\pi_{2n}(\underline{MU}_0)$ as multiplication by
$\lambda_{n}\in\Z$; that is,
$\beta_{MU|}(\theta)=(\lambda_n)_{n\geq 0}$.
 We give some simple examples of the congruences
satisfied by the $\lambda_n$
to illustrate how these arise.

The coefficient ring of $MU$ is
given by $MU^*=\Z[x_1,x_2,x_3,\dots]$, where $|x_i|=-2i$.
Using~\cite[Theorem 16.9]{bjw}, the free $MU^{\ast}$-module
$QMU_{\ast}(\underline{MU}_{0})$ has generators
$(b^{MU})^{\alpha}\eta_{R}(x)$,
where
$(b^{MU})^{\alpha}=(b_{1}^{MU})^{\alpha_{1}}(b_{2}^{MU})^{\alpha_{2}}\dots$,
for any finite sequence of non-negative integers
$(\alpha_{1},\alpha_{2},\dots)$, and $x\in MU^{-2|\alpha|}$ where
$|\alpha|=\sum_{i}\alpha_{i}$.

As noted in the proof of Proposition~\ref{faithfulaction}, the main relation in the Hopf ring shows that
rationally $QMU_{\ast}(\underline{MU}_{0})$ has $MU^*$-module generators
$e^{2h}\eta_{R}(x)$.

\begin{exs}
Consider $b_{2}^{MU}\eta_{R}(x_{1})\in QMU_{\ast}(\underline{MU}_0)$. We rewrite
$b_{2}^{MU}$ rationally as
$\frac{1}{2}(e^{4}\eta_{R}(x_{1})-x_{1}e^{2})$.
Then
    $$
    \overline{\theta}(b_{2}^{MU}\eta_{R}(x_{1}))
    = \frac{(\lambda_{2}-\lambda_{1})}{2}x_{1}^{2}\in MU^{\ast}.
    $$
This gives the congruence $\frac{\lambda_{2}-\lambda_{1}}{2}\in\Z$.
\smallskip

Now consider $b_{3}^{MU}\eta_{R}(x_{1})\in QMU_{\ast}(\underline{MU}_0)$.
The same procedure as above shows that
    $$
    \overline{\theta}(b_{3}^{MU}\eta_{R}(x_{1}))
    = \frac{(\lambda_{3}-3\lambda_{2}+2\lambda_1)}{6}x_{1}^{3}+\frac{(\lambda_3-\lambda_1)}{3}a_{2,1}x_1 \in MU^{\ast},
    $$
where $a_{2,1}\in MU^{-4}$.

So this gives us the two congruences
    $$
    \frac{\lambda_3-3\lambda_{2}+2\lambda_{1}}{6} \in\Z \qquad \text{and} \qquad
    \frac{(\lambda_3-\lambda_1)}{3}\in\Z.
    $$
\end{exs}

Notice that the three $MU$ congruences we have produced here are equivalent to the first
two $K$-theory ones: the first two are the same and the third is redundant.

It follows from the definitions
that we have an equality $\beta_{MU}(\mathcal{D}(MU))=S_{MU}$
and since $\beta_{MU}$ is an injective ring homomorphism, this
gives an isomorphism of
rings $\mathcal{D}(MU)\cong S_{MU}$.

\section{The centre}
\label{sec5}

Our goal is now to compare the solution sets $S_{KU}$ and $S_{MU}$
for the congruences coming from $K$-theory and from complex cobordism. We
will show that they are equal and this will allow us to prove the
main result about the centre of $\mathcal{A}(MU)$.

One inclusion follows directly from the existence of the map
$\iota:\mathcal{A}(KU)\hookrightarrow \mathcal{A}(MU)$.

\begin{prop}\label{SKinSMU}
We have the inclusion $S_{K}\subseteq S_{MU}$.
\end{prop}

\begin{proof}
Let $\lambda=(\lambda_{n})_{n\geq 0}\in S_{KU}\subset \prod_{n=0}^\infty \Z$.
Then $\lambda=\theta_{\ast}$ for some $\theta\in \mathcal{A}(KU)$ and
$\iota(\theta)\in \mathcal{A}(MU)$, with $(\iota(\theta))_{\ast}=\theta_{\ast}=\lambda$.
Hence, $\lambda\in S_{MU}$.
\end{proof}

To prove the reverse inclusion we consider the relationship between the Hopf rings
for $MU$ and $KU$. The standard map of ring spectra $\varphi: MU\rightarrow K$ induces
a map of Hopf rings $MU_*(\underline{MU}_*) \to KU_*(\underline{KU}_*)$.
Hence there is an induced ring map on indecomposables
$QMU_*(\underline{MU}_*) \to QKU_*(\underline{KU}_*)$, which we denote by
$\phi$.
We now choose the orientation class $x^{KU}$ for $K$-theory to be $\phi_*(x^{MU})$.
With this choice it is routine to check that $\phi(b_i^{KU})=b_i^{MU}$.

Fixing $\theta\in \mathcal{D}(MU)$, we consider
the $MU$ congruences satisfied by $\theta_*=(\lambda_n)_{n\geq 0}$.
It turns out, as the proof of the next proposition shows, that we obtain the $K$-theory
congruences among these
by considering the coefficient of $x_1^n$ in $\overline{\theta}(b_n^{MU}\eta_R(x_1))$.

\begin{prop}\label{SMUinSK}
$S_{MU}\subseteq S_{K}$.
\end{prop}

\begin{proof}
Let $\lambda\in S_{MU}$. Then there is a $\theta\in \mathcal{D}(MU)$ such
that $\theta_{\ast}=\lambda=(\lambda_{n})_{n\geq 0}
\in\prod_{n=0}^{\infty}\Z$. We define $V_{\lambda}: QMU_{\ast}(\underline{MU}_0)\to \Z$ by the composite
$\pi\overline{\theta}$ where $\pi : MU^{\ast}\rightarrow \Z$ is
defined to be the ring map determined by
    \begin{align*}
    x_1&\mapsto 1,\\
    x_i&\mapsto 0, \quad\text{for $i>1$}.
    \end{align*}
 Thus we have a commutative diagram
    $$
    \xymatrix{ QMU_{\ast}(\underline{MU}_0) \ar[rr]^-{\overline{\theta}}
    \ar[rrd]_-{V_{\lambda}} && MU^{\ast} \ar[d]^-{\pi} \\
    && \Z \\
    }
    $$
and the diagonal map takes $x\in QMU_{\ast}(\underline{MU}_0)$ to some rational
linear combination of the $\lambda_{i}$, which the $MU$ congruences tell
us is in $\Z$.

It is easy to check that we can factorize $V_{\lambda}$ as
$\pi_{\lambda}\widetilde{\phi}$ where
    $$
    \widetilde{\phi}:QMU_{\ast}(\underline{MU}_0)\rightarrow\textrm{Im}(\phi)
    $$
is the map given by restricting the range of
    $$
    \phi:QMU_{\ast}(\underline{MU}_0)\rightarrow QKU_{\ast}(\underline{KU}_0),
    $$
and
    $$
    \pi_{\lambda}:\textrm{Im}(\phi)\rightarrow\Z
    $$
is the $\Q$-linear map determined by
    $$
    u^{a}e^{2b}v^{b}\mapsto\lambda_{b}.
    $$

Now
    $$
    \widetilde{\phi}\left(b_{n}^{MU}\eta_{R}(x_{1})\right)=
    b_{n}^{KU}\eta_R(u)=b_{n}^{KU}v.
    $$

So
    \begin{align*}
    V_{\lambda}(b_{n}^{MU}\eta_{R}(x_{1}))
    &= \pi_{\lambda}\widetilde{\phi}(b_{n}^{MU}\eta_{R}(x_{1})) \\
    & = \pi_{\lambda}\left(b_{n}^{KU}v\right) \\
    & = C_{n}\cdot\lambda.
    \end{align*}

But $V_{\lambda}(b_{n}^{MU}\eta_{R}(x_{1}))=
\pi\bar{\theta}(b_{n}^{MU}\eta_{R}(x_{1}))\in\Z$. So
$C_{n}\cdot\lambda\in\Z$ for all $n\geq 0$ and thus $\lambda\in S_{K}$.

Hence $S_{MU}\subseteq S_{K}$.
\end{proof}

\begin{thm}\label{mainthm}
The image of the injective ring homomorphism $\iota:\mathcal{A}(KU)\hookrightarrow \mathcal{A}(MU)$
is the centre $Z(\mathcal{A}(MU))$.
\end{thm}

\begin{proof}
We now have the following commutative diagram, where
both vertical arrows are given by sending operations to their actions
on homotopy.
    $$
    \xymatrix{
    \mathcal{A}(KU) \ar[r]^{\cong}_{\iota}
    \ar[d]_{\cong} & \textrm{Im}(\iota)
    \ar@{^{(}->}[rr]^-{\textrm{by Lemma~\ref{imincentre}}} && Z(\mathcal{A}(MU))
    \ar@{^{(}->}[rr]^-{\textrm{by Lemma~\ref{centreindiag}}} && \mathcal{D}(MU)
    \ar[d]_-{\cong}  \\
    S_{KU} \ar[rrrrr]^-{=}_-{\genfrac{}{}{0pt}{}
    {\textrm{by Propositions~\ref{SKinSMU}}}{\textrm{and \ref{SMUinSK}}}} &&&&& S_{MU}
    }
    $$
It follows that the two inclusions on the
top line of the diagram must be equalities and hence
$\textrm{Im}(\iota)=Z(\mathcal{A}(MU))=\mathcal{D}(MU)$.
\end{proof}

\section{The split case}
\label{sec6}

Let $p$ be an odd prime. In this section we give
the analogue of Theorem~\ref{mainthm}
in the split $p$-local setting, that is with the Adams summand $G$ and 
Brown-Peterson theory $BP$ in place of $KU$ and $MU$.

Most of the steps in the proof follow those given earlier in the non-split
setting. To get started we need to know that
we can express all operations in $\mathcal{A}(G)$
in terms of Adams operations.

\begin{prop}
\label{Gbasis}
The topological ring $\mathcal{A}(G)$ may be identified with the collection
of infinite sums $\{\sum_{n=0}^\infty a_n \hat{\sigma}_n^{G}\, |\, a_n\in \Zpl \}$,
where each $\hat{\sigma}_n^{G}$ is a finite $\Zpl$-linear combination of the Adams
operations $\Psi^k_G$.
\end{prop}

\begin{proof}
We can obtain $\mathcal{A}(G)$ from $\mathcal{A}(KU)$ by applying the
Adams idempotent $e_0$. This idempotent operation acts on homotopy as the
identity on $\pi_{n}(\underline{KU}_0)$ if $n$ is a multiple of $2(p-1)$ and 
as zero otherwise. Thus we see that $e_0\Psi^k_{KU}=\Psi^k_G$.
We obtain the topological spanning set $\{e_0\sigma_n^{KU}\,|\,n\geq 0\}$ of $\mathcal{A}(G)$
from the topological basis of $\mathcal{A}(KU)$ given in Theorem~\ref{Kbasis}.
Each element is a finite linear combination of $G$ Adams operations.
Within this spanning set we can find a topological basis
$\{\hat{\sigma}_n\,|\, n\geq 0\}$
of $\mathcal{A}(G)$.
\end{proof}   

Explicit formulas  for a choice of such topological basis elements 
for $\mathcal{A}(G)$ are given
in~\cite[Chapter 4]{strong}, but we do not need these here.

\begin{thm}\label{mainthmBP}
There is an injective ring homomorphism
$\hat{\iota}:\mathcal{A}(G)\rightarrow \mathcal{A}(BP)$ such that
${\rm{Im}}(\hat{\iota})=Z(\mathcal{A}(BP))$.
\end{thm}

\begin{proof}[Outline proof of Theorem~\ref{mainthmBP}]

By Proposition~\ref{Gbasis}, we can identify
$\mathcal{A}(G)$ with
    $$
    \left\{ \sum_{n=0}^\infty a_n\hat{\sigma}_n^G\,|\,a_n\in\Zpl \right\}.
    $$
We define $\hat{\sigma}_n^{BP}\in\mathcal{A}(BP)$ in the obvious way, by replacing
$G$-Adams operations by the corresponding $BP$ ones, given by Proposition~\ref{adamsops}.
The method of the proof of Proposition~\ref{infsumsMU} shows that $\hat{\sigma}_n^{BP}\to 0$ as
$n\to\infty$ in the filtration topology of $\mathcal{A}(BP)$. Since
$\mathcal{A}(BP)$ is complete in this topology, we can define
$\hat{\iota}:\mathcal{A}(G)\to\mathcal{A}(BP)$ by
$\sum_{n=0}^{\infty}a_{n}\hat{\sigma}_{n}^{G}\mapsto
\sum_{n=0}^{\infty}a_{n}\hat{\sigma}_{n}^{BP}$.
This is an injective ring homomorphism, just as in the non-split case and
the same proof as for Lemma~\ref{imincentre} shows that
$\textrm{Im}(\hat{\iota})\subseteq Z(\mathcal{A}(BP))$.
By Lemma~\ref{centreindiag}, there
is an inclusion $Z(\mathcal{A}(BP))\subseteq \mathcal{D}(BP)$.

Recall that $\beta_{BP}$ is the injective ring homomorphism
    \begin{align*}
    \beta_{BP}:\mathcal{A}(BP) & \rightarrow \textrm{Ab}_*(\pi_*(\underline{BP}_0),\pi_*(\underline{BP}_0)) \\
    \theta & \mapsto \theta_{\ast}.
    \end{align*}
Restricting to $\mathcal{D}(BP)$, we have a map
    $$
    \beta_{BP|}:\mathcal{D}(BP)\to \prod_{n=0}^\infty \Zpl.
    $$
Consider the congruences satisfied by $\beta_{BP|}(\theta)=(\mu_n)_{n\geq 0}$,
where the operation $\theta$ acts on $\pi_{2(p-1)n}(\underline{BP}_0)$ as multiplication by $\mu_n$.
We have ${\beta_{BP}}(\mathcal{D}(BP))=S_{BP}$, and thus an isomorphism of rings
$\mathcal{D}(BP)\cong S_{BP}$.

For $E=BP$ or $G$, we write
$b_{(i)}^{E}=b_{p^i}^{E}$ and we recall that
all the other Hopf ring elements $b_i^{E}$ are redundant.
Let $(b^{E})^{\alpha}=(b_{(0)}^{E})^{\alpha_{0}}(b_{(1)}^{E})^{\alpha_{1}}\dots$ for a
finite integer sequence $\alpha=(\alpha_{0},\alpha_{1},\dots)$.

Then,
using~\cite[Theorem 16.11(a)]{bjw}, we find
that $QBP_{\ast}(\underline{BP}_0)$ is free as a $BP^{\ast}$-module and
it is generated by elements of the form $(b^{BP})^\alpha \eta_{R}(v)$ with $v\in BP^{-2|\alpha|}$ where
$|\alpha|=\sum\alpha_{i}$. 

So the $BP$ congruences come from $\overline{\theta}((b^{BP})^\alpha \eta_{R}(v))\in BP^*$, for
$v$ and $\alpha$ as above and we now compare the solution sets for the
$BP$ and $G$ congruences.

The inclusion $S_{G}\subseteq S_{BP}$ follows directly from the
existence of $\hat{\iota}$.

For the reverse inclusion, we consider the ring map 
$QBP_{\ast}(\underline{BP}_0)\rightarrow QG_{\ast}(\underline{G}_0)$
coming from the map of ring spectra $\hat{\phi}:BP\to G$. This
takes $b_{n}^{BP}$ to $b_{n}^{G}$. Write $G^*=\Zpl[\hat{u}, \hat{u}^{-1}]$
where $|\hat{u}|=2(p-1)$.
The elements
$(b^{G})^\alpha\hat{v}^{i}$ span $QG_{\ast}(\underline{G}_0)$ as a $G^{\ast}$-module,
where $\hat{v}=\eta_R(\hat{u})$
and $i\in\Z$ satisfies $\sum \alpha_j=i(p-1)$.
Let $\{\hat{f}_n\,|\, n\geq 0\}$ be a $\Zpl$-basis of $QG_0(\underline{G}_0)$.
Then we can express each $\hat{f}_{n}$ as a
$G^{\ast}$-linear combination of the $(b^G)^\alpha \hat{v}^i$. This means that, up to some
shift by a power of $\hat{u}$, each $\hat{f}_{n}$ is in
the image of the map from $QBP_{\ast}(\underline{BP}_0)$. The analogous proof to
that in Lemma~\ref{SMUinSK} now shows that $S_{BP}\subseteq S_{G}$.

Therefore we have the following commutative diagram, where both vertical maps send
operations to their actions on homotopy.
    $$
    \xymatrix{
    \mathcal{A}(G) \ar[r]^-{\cong}_-{\hat{\iota}} \ar[d]_-{\cong} &
    \textrm{Im}(\hat{\iota}) \ar@{^{(}->}[rr] && Z(\mathcal{A}(BP))
    \ar@{^{(}->}[rr] && \mathcal{D}(BP)
    \ar[d]_-{\cong} \\
    S_{G} \ar[rrrrr]^-{=} &&&&& S_{BP}
    }
    $$

It follows that $\textrm{Im}(\hat{\iota})=Z(\mathcal{A}(BP))$.
\end{proof}

\end{document}